\documentclass{amsart}

\usepackage{fontspec}
\usepackage{booktabs}
\usepackage{hyperref}

\usepackage{graphicx}
\usepackage{tikz}
\usepackage{tabularx}

\newtheorem{theorem}{Theorem}[section]
\newtheorem{lemma}[theorem]{Lemma}

\theoremstyle{remark}
\newtheorem{remark}[theorem]{Remark}

\newcommand{\Zn}{\mathbb{Z}[n]}
\newcommand{\Z}{\mathbb{Z}}
\newcommand{\Q}{\mathbb{Q}}

\title{Purely cosmetic surgeries on knots with low-span Jones polynomials}
\author{Kazuhiro Ichihara}
\address{Department of Mathematics, College of Humanities and Sciences, Nihon University, 3-25-40 Sakurajosui, Setagaya-ku, Tokyo 156-8550, JAPAN}
\email{ichihara.kazuhiro@nihon-u.ac.jp}

\dedicatory{Dedicated to Professor Yasushi Yamashita on his 60th birthday.}



\thanks{This work was supported by JSPS KAKENHI Grant Number JP26K22260.}

\date{\today}

\begin{document}

\begin{abstract}
In this paper, we show that a knot with a nontrivial Jones polynomial of span at most 11 does not admit purely cosmetic surgeries.
\end{abstract}

\maketitle

\section{Introduction}

Let $K$ be a nontrivial knot in the 3-sphere $S^3$.
If there is an orientation-preserving homeomorphism between the 3-manifolds obtained by Dehn surgeries on $K$ along distinct slopes, such a pair of surgeries is said to be \emph{purely cosmetic}.
The Purely Cosmetic Surgery Conjecture states that no nontrivial knot in $S^3$ admits purely cosmetic surgeries (\cite[Problem 1.12(c)]{K3List}).

The main result of this paper is as follows.

\begin{theorem}\label{thm:main}
If a knot $K$ in $S^3$ has a nontrivial Jones polynomial $V_K(t)$ and $\operatorname{span}(V_K(t)) \le 11$, then $K$ does not admit purely cosmetic surgeries.
\end{theorem}

There are infinitely many knots that satisfy the assumption of the theorem, that is, knots with a nontrivial Jones polynomial $V_K(t)$ whose $\operatorname{span}(V_K(t)) $ is at most 11.
For example, Kanenobu \cite[p.160]{Kanenobu} gives an infinite family of knots $K_{n,-n}$ with the same Jones polynomial of span 8. 
See Figure~\ref{fig:K2} for the knot $K_{2,-2}$.
The Jones polynomial of these knots is the following:
\[
V_{K_{n,-n}}(t)=
( t^2 -t + 1 - t^{-1} + t^{-2} )^2 
= t^{-4} - 2t^{-3} + 3t^{-2} - 4t^{-1} + 5 - 4t + 3t^2 - 2t^3 + t^4.
\]

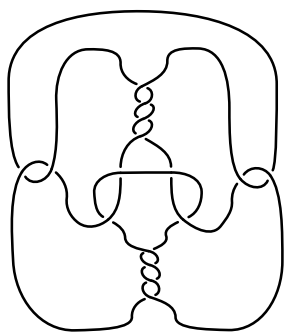
\begin{figure}[htb]
    \centering
    \resizebox{0.3\textwidth}{!}{
\definecolor{linkcolor0}{rgb}{0, 0, 0}
\begin{tikzpicture}[line width=2.6, line cap=round, line join=round]
  \begin{scope}[color=linkcolor0]
    \draw (9.46, 5.80) .. controls (9.60, 6.12) and (9.60, 7.74) .. 
          (9.60, 8.84) .. controls (9.60, 10.73) and (7.17, 11.38) .. 
          (4.87, 11.40) .. controls (2.59, 11.43) and (0.14, 10.91) .. 
          (0.14, 8.95) .. controls (0.14, 7.90) and (0.14, 6.44) .. (0.49, 5.91);
    \draw (0.73, 5.57) .. controls (0.95, 5.24) and (1.50, 5.38) .. (1.66, 5.85);
    \draw (1.66, 5.85) .. controls (1.88, 6.51) and (1.85, 7.46) .. 
          (1.82, 8.27) .. controls (1.79, 9.16) and (2.16, 10.05) .. 
          (2.93, 10.06) .. controls (3.49, 10.07) and (4.10, 10.04) .. 
          (4.09, 9.57) .. controls (4.08, 9.23) and (4.33, 8.96) .. (4.65, 8.84);
    \draw (5.04, 8.69) .. controls (5.27, 8.61) and (5.25, 8.25) .. (4.97, 8.16);
    \draw (4.97, 8.16) .. controls (4.61, 8.05) and (4.51, 7.74) .. (4.69, 7.68);
    \draw (5.09, 7.55) .. controls (5.30, 7.48) and (5.20, 7.12) .. (4.80, 7.00);
    \draw (4.80, 7.00) .. controls (4.35, 6.86) and (4.10, 6.40) .. (4.09, 5.92);
    \draw (4.09, 5.51) .. controls (4.09, 4.98) and (4.06, 4.40) .. (3.65, 4.06);
    \draw (3.65, 4.06) .. controls (3.35, 3.81) and (2.93, 3.72) .. 
          (2.61, 3.94) .. controls (2.33, 4.14) and (2.15, 4.47) .. 
          (2.19, 4.81) .. controls (2.23, 5.16) and (2.07, 5.49) .. (1.81, 5.72);
    \draw (1.50, 5.99) .. controls (1.23, 6.22) and (0.82, 6.07) .. (0.61, 5.74);
    \draw (0.61, 5.74) .. controls (0.16, 5.04) and (0.24, 3.67) .. 
          (0.30, 2.58) .. controls (0.37, 1.28) and (1.25, 0.13) .. 
          (2.48, 0.16) .. controls (3.37, 0.18) and (4.52, 0.21) .. 
          (4.49, 0.64) .. controls (4.48, 0.92) and (4.67, 1.16) .. (4.94, 1.26);
    \draw (5.32, 1.41) .. controls (5.53, 1.49) and (5.49, 1.85) .. (5.18, 1.91);
    \draw (5.18, 1.91) .. controls (4.83, 1.98) and (4.75, 2.32) .. (4.94, 2.38);
    \draw (5.32, 2.53) .. controls (5.49, 2.59) and (5.41, 2.89) .. (5.08, 2.95);
    \draw (5.08, 2.95) .. controls (4.71, 3.03) and (4.26, 3.12) .. 
          (4.25, 3.43) .. controls (4.24, 3.67) and (4.05, 3.84) .. (3.83, 3.96);
    \draw (3.47, 4.16) .. controls (3.21, 4.31) and (3.18, 4.68) .. 
          (3.16, 5.01) .. controls (3.13, 5.44) and (3.60, 5.71) .. (4.09, 5.71);
    \draw (4.09, 5.71) .. controls (4.69, 5.72) and (5.28, 5.72) .. (5.88, 5.72);
    \draw (5.88, 5.72) .. controls (6.40, 5.73) and (6.93, 5.52) .. 
          (6.95, 5.05) .. controls (6.96, 4.68) and (6.82, 4.31) .. (6.49, 4.15);
    \draw (6.11, 3.97) .. controls (5.88, 3.86) and (5.62, 3.71) .. 
          (5.70, 3.48) .. controls (5.77, 3.27) and (5.51, 3.13) .. (5.27, 3.03);
    \draw (4.89, 2.88) .. controls (4.72, 2.81) and (4.82, 2.54) .. (5.13, 2.46);
    \draw (5.13, 2.46) .. controls (5.48, 2.37) and (5.56, 2.05) .. (5.37, 1.98);
    \draw (4.99, 1.84) .. controls (4.79, 1.77) and (4.83, 1.44) .. (5.13, 1.33);
    \draw (5.13, 1.33) .. controls (5.55, 1.19) and (6.03, 1.03) .. 
          (6.03, 0.62) .. controls (6.03, 0.21) and (7.09, 0.19) .. 
          (7.91, 0.17) .. controls (9.06, 0.15) and (9.79, 1.32) .. 
          (9.80, 2.57) .. controls (9.80, 3.62) and (9.81, 5.00) .. (9.37, 5.62);
    \draw (9.37, 5.62) .. controls (9.14, 5.94) and (8.67, 5.96) .. (8.44, 5.65);
    \draw (8.20, 5.32) .. controls (8.07, 5.14) and (8.01, 4.93) .. 
          (8.02, 4.71) .. controls (8.05, 4.38) and (7.83, 4.10) .. 
          (7.63, 3.84) .. controls (7.33, 3.45) and (6.75, 3.69) .. (6.30, 4.06);
    \draw (6.30, 4.06) .. controls (5.89, 4.40) and (5.88, 4.98) .. (5.88, 5.52);
    \draw (5.87, 5.93) .. controls (5.87, 6.40) and (5.40, 6.66) .. (4.98, 6.90);
    \draw (4.62, 7.10) .. controls (4.44, 7.21) and (4.57, 7.50) .. (4.89, 7.61);
    \draw (4.89, 7.61) .. controls (5.25, 7.74) and (5.36, 8.05) .. (5.16, 8.11);
    \draw (4.77, 8.22) .. controls (4.55, 8.28) and (4.56, 8.63) .. (4.85, 8.76);
    \draw (4.85, 8.76) .. controls (5.27, 8.97) and (5.75, 9.20) .. 
          (5.74, 9.65) .. controls (5.73, 10.09) and (6.32, 10.10) .. 
          (6.83, 10.09) .. controls (7.70, 10.07) and (7.93, 9.00) .. 
          (7.93, 8.00) .. controls (7.93, 7.12) and (7.93, 5.94) .. (8.32, 5.49);
    \draw (8.32, 5.49) .. controls (8.62, 5.14) and (9.12, 5.09) .. (9.28, 5.43);
  \end{scope}
\end{tikzpicture}}
    \caption{Kanenobu knot $K_{2,-2}$}
    \label{fig:K2}
\end{figure}

In the proof of the theorem, we use known necessary conditions on the Jones polynomial for a knot to have purely cosmetic surgeries.
In the range of span at most 11, the only Jones polynomial that satisfies these conditions is the trivial one $V_K(t) = 1$, but for span 12 there are nontrivial solutions that satisfy the conditions.
We describe this result in Section \ref{sec:span12}.

The paper is organized as follows.
In Section~\ref{sec:poly-list}, we review basic evaluation conditions that the Jones polynomial of a knot satisfies, and the classification for span at most 5 by Kanenobu--Kishimoto--Sumi \cite{KanenobuKishimotoSumi}.
In Section~\ref{sec:small-span}, we combine these with the conditions that come from the existence of purely cosmetic surgeries to prove Theorem~\ref{thm:main} by hand in the case of span at most 5.
In Sections~\ref{sec:sympy-rigor} and~\ref{sec:algorithm}, we set up a general framework (Theorem~\ref{thm:poly}) for span 6 or more, and then run an exhaustive search with SymPy \cite{Meurer2017} for spans from 6 to 11 to complete the proof of Theorem~\ref{thm:main}.
Finally, in Section~\ref{sec:span12}, we see that for span $12$ there actually exist nontrivial polynomials that satisfy the conditions. 

\section{Laurent polynomials that can be Jones polynomials}\label{sec:poly-list}

The Jones polynomial $V(t) \in \Z[t^{\pm 1}]$ of a knot $K \subset S^3$ satisfies the following relations at specific points $t$ (see \cite[Section 2]{KanenobuKishimotoSumi}):
\begin{align}
V(1) &= 1, \label{eq:j1} \\
V'(1) &= 0, \label{eq:j2} \\
V(e^{2\pi i / 3}) &= 1, \label{eq:j3} \\
V(i) &= \pm 1. \label{eq:j4}
\end{align}
Here $V'(1)$ is the value of the first derivative at $t=1$.
In \cite{KanenobuKishimotoSumi}, further conditions, at $t=e^{\pi i/3}$ and $t=-1$, are also used, but we omit them since we do not use them in this paper.

For a Laurent polynomial $f(t)$, the \emph{span} of $f(t)$ is the difference between the largest and the smallest exponents of its nonzero terms.
According to Kanenobu--Kishimoto--Sumi \cite[Theorem 1.1]{KanenobuKishimotoSumi}, if the Jones polynomial $f(t)=V_K(t)$ of a knot $K$ is written as $f(t) = t^r P(t)$ with $r \in \Z$ and $P(t) = \sum_{k=0}^N c_k t^k \in \Z[t]$, and it is nontrivial with span at most 5, then it is one of the following.
Note that the span of $f(t)$ is at most the degree of $P(t)$, and each $P(t)$ below satisfies $c_0 \ne 0$, so the span of $f(t)$ is equal to the degree of $P(t)$.

When the span is at most 4, a nontrivial candidate is one of the following (to be precise, there is also $f_2(t^{-1})$; note that $f_3(t^{-1})=f_3(t)$):
\[
f_2(t)=t(1+t^2-t^3), \qquad 
f_3(t)=t^{-2}(1-t+t^2-t^3+t^4).
\]
When the span is 5, a nontrivial candidate is one of the following five, or $f_k(t^{-1})$ for one of them:
\begin{align*}
f_4(t) &= t^2 ( 1 + t^2 - t^3 + t^4 - t^5 ) 
    = t^2 + t^4 - t^5 + t^6 - t^7, \\
f_5(t) &= t ( 1 - t + 2t^2 - t^3 + t^4 - t^5 ) 
    = t - t^2 + 2t^3 - t^4 + t^5 - t^6, \\
f_6(t) &= t^3 ( 1 + t^2 - t^5) 
    = t^3 + t^5 - t^8, \\
f_7(t) &= t^4( 2 - t + 2t^2 - 2t^3 + t^4 - t^5 ) 
    = 2t^4 - t^5 + 2t^6 - 2t^7 + t^8 - t^9, \\
f_8(t) &= t^8 ( 2 + 2t^2 - 2t^3 + t^4 - 2t^5)
    = 2t^8 + 2t^{10} - 2t^{11} + t^{12} - 2t^{13}.
\end{align*}

The list for the case where $f(t)$ has span 6 is also given in \cite{KanenobuKishimotoSumi}, but
the detailed computation is omitted there, so in the later sections we treat the spans from 6 to 11.

\section{Jones polynomials and purely cosmetic surgeries}\label{sec:small-span}

Suppose that a knot $K \subset S^3$ admits purely cosmetic surgeries.
Then the Jones polynomial $f(t) = V_K(t)$ of $K$ must satisfy the following algebraic conditions:
\begin{align}
f''(1) &= f'''(1) = f''''(1)=0, \label{eq:cs1} \\
f(\zeta) &= 1 \quad (\zeta \text{ is a primitive 5th root of unity}). \label{eq:cs2}
\end{align}
The condition $f''(1)=0$ is based on Boyer--Lines \cite{BoyerLines}, $f'''(1)=0$ on Ichihara--Wu \cite{IchiharaWu}, $f''''(1)=0$ on Ito \cite{Ito} and Brady \cite{Brady}, and $f(\zeta)=1$ on Detcherry \cite{Detcherry} and Daemi--Lidman--Miller Eismeier \cite{DLME}.
In fact, Detcherry \cite[Theorem 1.4]{Detcherry} shows that either the slopes of a pair of purely cosmetic surgeries are of the form $\pm 1/(5k)$, or $f(\zeta)=1$ holds. 
On the other hand, by Daemi--Lidman--Miller Eismeier \cite[Corollary 1.4]{DLME}, the slopes of a pair of purely cosmetic surgeries are $\pm 2$, the genus of $K$ is $2$, and the Alexander polynomial is trivial. 
Since $\pm 2$ is not of the form $\pm 1/(5k)$, we have $f(\zeta)=1$.
Also, Brady \cite{Brady} obtains $f''''(1)=0$ by putting these conditions on the slopes and the Alexander polynomial into the relation of Ito \cite{Ito}.

Moreover, the following lemma holds.

\begin{lemma}\label{lem:arf_i}
If a knot $K$ admits purely cosmetic surgeries, then its Jones polynomial $f(t)=V_K(t)$ satisfies $f(i) = 1$.
\end{lemma}

\begin{proof}
The evaluation \eqref{eq:j4} follows from $V(i) = (-1)^{\operatorname{Arf}(K)}$, where $\operatorname{Arf}(K)$ is the Arf invariant of the knot $K$.
It is known that $\operatorname{Arf}(K) \equiv a_2(K) \pmod 2$ and $a_2(K) = -V''(1)/6$, where $a_2(K)$ is the second coefficient of the Conway polynomial. 
See, for example, \cite{KanenobuKishimotoSumi}.
Since $K$ admits purely cosmetic surgeries, $a_2(K) = -f''(1)/6 = 0$ by \eqref{eq:cs1}.
Therefore $\operatorname{Arf}(K) \equiv a_2(K) \equiv 0 \pmod 2$, and we get $f(i) = (-1)^{\operatorname{Arf}(K)} = (-1)^0 = 1$.
\end{proof}

We now evaluate these conditions~\eqref{eq:cs1},~\eqref{eq:cs2} for the candidates given in Section~\ref{sec:poly-list}.
For the nontrivial Jones polynomial candidates of span at most 5 (the seven polynomials $f_2, \dots, f_8$), the values $f(\zeta)$ at a primitive 5th root of unity $\zeta$ are as shown in Table~\ref{tab:zeta_evaluation}.

\begin{table}[htbp]
  \centering
  \caption{Evaluation by substitution at a primitive 5th root of unity $\zeta$ for each Laurent polynomial ($f(\zeta)$)}
  \label{tab:zeta_evaluation}
  \vspace{0.5em}
  \begin{tabular}{c c c}
    \toprule
    $f_k$ & Polynomial $f(t)$ & $f(\zeta)$ \\
    \midrule
    $f_2$ & $t(1 + t^2 - t^3)$ & $1 + 2\zeta + \zeta^2 + 2\zeta^3$ \\
    $f_3$ & $t^{-2}(1 - t + t^2 - t^3 + t^4)$ & $2 + 2\zeta^2 + 2\zeta^3$ \\
    \midrule
    $f_4$ & $t^2 + t^4 - t^5 + t^6 - t^7$ & $-2 - \zeta^2 - \zeta^3$ \\
    $f_5$ & $t - t^2 + 2t^3 - t^4 + t^5 - t^6$ & $2 + \zeta + 3\zeta^3$ \\
    $f_6$ & $t^3 + t^5 - t^8$ & $1$ \\
    $f_7$ & $2t^4 - t^5 + 2t^6 - 2t^7 + t^8 - t^9$ & $-2 + \zeta - 3\zeta^2$ \\
    $f_8$ & $2t^8 + 2t^{10} - 2t^{11} + t^{12} - 2t^{13}$ & $2 - 2\zeta + \zeta^2$ \\
    \bottomrule
  \end{tabular}
\end{table}

Here, the minimal polynomial of a primitive 5th root of unity $\zeta$ is
\[
\Phi_5(t) = t^4 + t^3 + t^2 + t + 1,
\]
and $\{1, \zeta, \zeta^2, \zeta^3\}$ is a basis of the algebraic number field $\Q(\zeta)$ over $\Q$.
Therefore, by Table~\ref{tab:zeta_evaluation}, the only one that satisfies $f(\zeta) = 1$ is $f_6(t) = t^3 + t^5 - t^8$ (the Jones polynomial of the knot $8_{19}$).
Note that for $f_k(t^{-1})$, which corresponds to the mirror image of the knot ($t \mapsto t^{-1}$), we get the same results:
since $f_k$ has integer coefficients and $\zeta^{-1}=\zeta^4$ is a conjugate of $\zeta$ over $\Q$, $f_k(\zeta^{-1})=1$ holds if and only if $f_k(\zeta)=1$;
and since $f_k'(1)=0$, we have $f_k''(1) = (f_k(t^{-1}))''|_{t=1}$.

For $f_6(t) = t^3 + t^5 - t^8$, computing the value of the second derivative at $t=1$ gives
\[
f_6''(1) = \left.(6t + 20t^3 - 56t^6)\right|_{t=1} = 6 + 20 - 56 = -30 \neq 0,
\]
and so it does not satisfy the condition $f''(1) = 0$.

Therefore, in the range $\operatorname{span}(f(t)) \le 5$, there is no nontrivial Jones polynomial that a knot admitting purely cosmetic surgeries can have.
This proves Theorem~\ref{thm:main} in the case where $\operatorname{span}(V_K(t)) \le 5$.

\section{Setting of the problem and ensuring rigor with a computer algebra system}\label{sec:sympy-rigor}

As stated in the previous section, the Jones polynomial of a knot admitting purely cosmetic surgeries is subject to the condition $f(i) = 1$ from Lemma~\ref{lem:arf_i}, and the condition \eqref{eq:cs2} $f(\zeta) = 1$ from Detcherry \cite{Detcherry} and Daemi--Lidman--Miller Eismeier \cite{DLME}, in addition to the conditions \eqref{eq:j1}, \eqref{eq:j2}, \eqref{eq:j3} and \eqref{eq:cs1}. 

From here on, for the Laurent polynomial with integer coefficients
$f(t) = t^r P(t)$, where $r \in \Z$ and $P(t) = \sum_{k=0}^N c_k t^k \in \Z[t]$, we study the solution space when these conditions are imposed. 
Based on exact symbolic computation with the computer algebra system SymPy \cite{Meurer2017}, we check rigorously that every Laurent polynomial with integer coefficients and span at most 11 that satisfies these conditions is equal to the trivial solution $f(t)=1$.

Our goal is to prove the following theorem. 

\begin{theorem}\label{thm:poly}
Suppose that a Laurent polynomial
\[f(t)=t^rP(t) \quad \text{ with } r \in \Z , \ P(t)=\sum_{k=0}^N c_k t^k\in\Z[t] \]
satisfies the following value and derivative conditions at $1$:
\begin{align*}
  \text{(A1)} \quad & f(1)      = 1, \\
  \text{(A2)} \quad & f'(1)     = 0, \\
  \text{(A3)} \quad & f''(1)    = 0, \\
  \text{(A4)} \quad & f'''(1)   = 0, \\
  \text{(A5)} \quad & f''''(1)  = 0,
\end{align*}
and the following conditions at roots of unity:
\begin{align*}
  \text{(B)}  \quad & f(i)      = 1
  \qquad (i \text{ is a primitive 4th root of unity}), \\
  \text{(C)}  \quad & f(\omega) = 1
  \qquad (\omega \text{ is a primitive 3rd root of unity}), \\
  \text{(D)}  \quad & f(\zeta)  = 1
  \qquad (\zeta \text{ is a primitive 5th root of unity}).
\end{align*}
Then, if $\operatorname{span}(f)\le 11$, we have $f(t)=1$.
\end{theorem}

Note that any Laurent polynomial $f(t)$ with $\operatorname{span}(f)=s$ can be written as $f(t)=t^rP(t)$ with $\deg P=s$ and $c_0\ne0$.
Hence the search with the degree bound $N$ covers all Laurent polynomials with span at most $N$.
In particular, the search for $N=11$ alone covers all Laurent polynomials with span at most $11$.
In this paper, we carry out the search for each $N=6,\dots,11$ to see how the structure of the solutions changes with $N$.

\subsection{Target Laurent polynomials and search conditions}

In this subsection, we give some preparations for the proof of Theorem~\ref{thm:poly}.

We consider the Laurent polynomial
\[
f(t) = t^r P(t)
\]
given by $r \in \Z$ and the polynomial with integer coefficients
\[
P(t) = c_0 + c_1 t + \cdots + c_N t^N \in \Z[t].
\]
Here $N$ is the upper bound on the degree allowed for the polynomial $P(t)$,
and this includes the case where the actual degree of $P(t)$ is less than $N$.
Therefore, $N$ gives an upper bound on the span of $f(t)$.

\subsubsection{Conditions at roots of unity}
Through algebraic relations in the quotient ring
\[
\Z[t]/(\Phi_m(t)),
\]
where $\Phi_m(t)$ denotes the $m$-th cyclotomic polynomial, the root-of-unity conditions (B),(C),(D) are reduced to a system of linear equations in $c_0,\dots,c_N$ as follows.
First, since $\Phi_m(t)$ is the minimal polynomial of a primitive $m$-th root of unity $\zeta_m$ over $\Q$,
the map $t\mapsto\zeta_m$ induces an isomorphism $\Z[t]/(\Phi_m(t))\cong\Z[\zeta_m]$.
Hence $\Z[\zeta_m]$ is a free module of rank $\varphi(m)$ over $\Z$ with basis
$\{1,\zeta_m,\dots,\zeta_m^{\varphi(m)-1}\}$, where $\varphi$ is Euler's totient function.
Since $\zeta_m^m=1$, the condition $f(\zeta_m)=1$ is equivalent to
\[
P(\zeta_m)=\zeta_m^{-r},
\]
whose right-hand side depends only on $r \bmod m$.
Reducing both sides of this equality in this basis by $\Phi_m(\zeta_m)=0$ and comparing
the coefficients, we obtain exactly $\varphi(m)$ linear equations in $c_0,\dots,c_N$
with integer coefficients, for each residue class of $r$ modulo $m$.
Note that this condition does not depend on the choice of the primitive $m$-th root of unity $\zeta_m$:
since $f$ has integer coefficients, if $f(\zeta_m)=1$ for one primitive $m$-th root of unity $\zeta_m$, then the same holds for all of them, as they are conjugate over $\Q$.
Since $\varphi(3)=2$, $\varphi(4)=2$, $\varphi(5)=4$,
condition (B) ($4$th roots of unity) gives $2$ linear equations, (C) ($3$rd roots of unity) gives $2$,
and (D) ($5$th roots of unity) gives $4$.

\subsubsection{Conditions at \texorpdfstring{$t=1$}{t=1}}
The conditions (A1) and (A2) also give linear equations.
Since $f(1)=P(1)$, the condition (A1) is
\[
c_0+c_1+\cdots+c_N=1 .
\]
Since $f'(t)=rt^{r-1}P(t)+t^rP'(t)$, we have $f'(1)=rP(1)+P'(1)$,
which contains the products of $r$ and $c_k$.
However, using $P(1)=1$ from (A1), the condition (A2) becomes
\[
r+\sum_{k=1}^{N} k\,c_k=0 ,
\]
which is linear in $c_0,\dots,c_N$ and $r$.
Adding these $2$ equations, the number of linear equations obtained from (A1),(A2),(B),(C),(D) is in total
\[
2+2+2+4=10.
\]

\subsubsection{Reduction modulo \texorpdfstring{$60$}{60}}
For $m=3,4,5$, let $r_m\in\{0,1,\dots,m-1\}$ denote the residue of $r$ modulo $m$.
As we saw above, the conditions (B), (C), (D) depend on $r$ only through
$r_4$, $r_3$, $r_5$, respectively.
Since $\operatorname{lcm}(3,4,5)=60$, they depend on $r$ only through $r \bmod 60$.
Thus we write
\[
r=60n+R, \qquad n\in\Z, \quad 0\le R<60 .
\]
By the Chinese remainder theorem, $R$ is the unique integer with $0\le R<60$ and
\[
R\equiv r_3 \pmod 3, \qquad R\equiv r_4 \pmod 4, \qquad R\equiv r_5 \pmod 5,
\]
and the map $R\mapsto(r_3,r_4,r_5)$ gives a bijection between $\{0,1,\dots,59\}$ and
$\{0,1,2\}\times\{0,1,2,3\}\times\{0,1,2,3,4\}$.
Therefore, by checking all $3\cdot4\cdot5=60$ triples $(r_3,r_4,r_5)$,
equivalently all $60$ values of $R$, together with $n\in\Z$, we cover all $r\in\Z$.
In particular, the condition (A2) becomes
\[
60n+R+\sum_{k=1}^{N} k\,c_k=0,
\]
which is linear in $c_0,\dots,c_N$ and $n$.

\subsection{Algebraic computation with SymPy and its mathematical rigor}\label{subsec:AlgComp}

The analysis and exhaustive search of the solution space in this paper are carried out by algorithms implemented with SymPy \cite{Meurer2017},
a Python package for algebraic computation.
The search for each $N$ is implemented by the corresponding SymPy script
\cite{Span6,Span7,Span8,Span9,Span10,Span11}.

The mathematical rigor of the computation rests on the following four points.

\begin{enumerate}

  \item
  \textbf{Exact computation over $\Q$}:
  We do not use approximate numerical computation or numerical search with floating-point numbers.
  The construction of the linear systems, operations on coefficient matrices, Gaussian elimination,
  and polynomial operations are done as exact symbolic computation over the field of rational numbers $\Q$.
  In the final direct check at special values too, the needed algebraic numbers are handled exactly.

  \item\label{itm:4.2(2)}
  \textbf{Covering all integer roots by the rational root theorem}:
  In the search for integer roots of a polynomial $g(n)\in\Zn$ with integer coefficients,
  we do not do a numerical search that tries integers in a finite range.
  For a nonzero polynomial with integer coefficients whose constant term is not $0$,
  we use the fact that, by the rational root theorem, every integer root divides the constant term.
  When the constant term is $0$, we divide $g(n)$ by the highest power of $n$ dividing it,
  add $0$ to the candidates, and then apply the rational root theorem to the remaining polynomial.
  We then check all of the resulting candidates by exact substitution, so
  we can find all integer roots,
  without missing any integer root and without accepting any extra root.

  \item
  \textbf{Exact checks of rank, determinant, and so on for the finitely many residue classes}:
  In the code, conditions such as the rank of the coefficient matrix, the determinant, the degree of polynomials,
  and membership in $\Zn$ are checked at runtime.
  For example, in the case $N=8$, we check that the determinant of the $10\times10$ coefficient matrix
  of the linear system in $c_0,\dots,c_8,n$ is exactly $1800$ for all 60 residue classes.
  Also, for $N=10$ we check that the coefficient of $c_{10}$ in $f''(1)$, after substituting the solution of the linear stage, is
  $60$ for all residue classes, and for $N=11$ that the determinant of the
  $2\times2$ coefficient matrix for (A3),(A4) is $10800$ for all residue classes.

  \item
  \textbf{Higher derivatives based on the Leibniz formula, and independent checks by direct differentiation}:
  To evaluate the higher derivatives $f^{(k)}(1)$, we do not expand them by hand for each order;
  we use a general computation based on the Leibniz formula.
  Concretely, we use the fact that, with the falling factorial
  \[
  (r)_j=r(r-1)\cdots(r-j+1),
  \]
  we have
  \[
  f^{(k)}(1)
  =
  \sum_{j=0}^{k}
  \binom{k}{j}
  (r)_j
  P^{(k-j)}(1) .
  \]

  Moreover, for each solution obtained, we do not reuse this formula;
  instead, we construct the actual Laurent polynomial
  \[
  f(t)=t^rP(t),
  \]
  differentiate it directly with the SymPy function \texttt{diff},
  and check the conditions at $t=1,i,\omega,\zeta$ again independently.

\end{enumerate}

\section{Search for solutions and the algorithm}\label{sec:algorithm}

In this section, for the Laurent polynomials
$f(t) = t^r P(t)$ with $r \in \Z$ and $P(t) = \sum_{k=0}^N c_k t^k \in \Z[t]$ 
that satisfy the conditions (A1)--(A5) and (B),(C),(D), we explain in detail the process (algorithm) of the search with the computer algebra system SymPy \cite{Meurer2017}.

The procedures described below for each $N=6,\dots,11$ are implemented as the corresponding SymPy scripts
\cite{Span6,Span7,Span8,Span9,Span10,Span11}, and
what follows is a summary of their content.
In particular, the concrete values of ranks, determinants, and coefficients, the degrees of polynomials, and membership in $\Zn$ are not mere claims; they are checked at runtime by \texttt{assert} statements in each script. Numbers such as the numbers of integer solutions are given as the output of each script.

\bigskip

Here we make a remark on the criteria for classifying the solutions in this search.

A search with the degree bound $N$ includes all solutions whose span is at most $N$.
Therefore, lower-span solutions with $c_N=0$ are also included in the search.

In terms of the coefficients, the actual span of $f(t)$ is given by 
\[
\operatorname{span}(f(t)) = \max\{k\mid c_k\ne0\} - \min\{k\mid c_k\ne0\} .
\]

The criteria for classifying the solutions in this search are as follows.

\begin{itemize}
  \item
  \textbf{Trivial solution}:
  a solution with $f(t)=1$. This corresponds to a monomial of span $0$. Note that any solution of span $0$ is trivial by (A1) and (A2).

  \item
  \textbf{Genuine span-$N$ solution}:
  a solution that satisfies $c_0\ne0$ and $c_N\ne0$, that is, whose actual span is exactly $N$.

  \item
  \textbf{Lower-span solution}:
  a solution that is nontrivial and whose actual span is less than $N$.
\end{itemize}

\subsection{Basic plan of the solution algorithm}

The search for solutions of the system of equations is done in the following two stages.

\begin{enumerate}

  \item
  \textbf{Linear stage}:
  For the 10 linear equations in total obtained from (A1), (A2), (B), (C), (D),
  we include $n$ as one of the symbolic unknowns together with $c_0 , \dots , c_N$, and analyze the solutions with the SymPy function \texttt{linsolve}.

  \item
  \textbf{Narrowing down the integer solutions by the higher derivative conditions}:
  For the free variables not determined in the linear stage,
  we apply (A3), (A4), (A5) one after another as far as needed.
  When the remaining condition becomes a polynomial equation in $n$ with integer coefficients, we find its integer roots exhaustively by the rational root theorem (see item~(\ref{itm:4.2(2)}) in Section~\ref{subsec:AlgComp}).

\end{enumerate}

More precisely,
\begin{itemize}
  \item for $N=6,7,8$ we use up to (A3).
  \item for $N=9,10$ we use up to (A4).
  \item for $N=11$ we use up to (A5).
\end{itemize}

As we will see below, for $N=6,7,8$, all solutions that pass (A3) are $f(t)=1$, so (A4),(A5) are satisfied automatically.
Also, for $N=9,10$, all solutions that finally pass (A4) are
$f(t)=1$, so (A5) is also satisfied automatically.

\subsection{Search in the overdetermined and square systems (\texorpdfstring{$N=6,7,8$}{N=6,7,8})}

In the range $N\le8$, there are 10 linear equations,
and the number of unknowns, the coefficients $c_0,\dots,c_N$ together with $n$, is at most 10.
Therefore, overdetermined systems or square systems appear in the linear stage.

\subsubsection{The case \texorpdfstring{$N=6$}{N=6}: the concrete system of equations}

In the case $N=6$, we have
\[P(t)=c_0+c_1t+\cdots+c_6t^6\]
and we put $r=60n+R$ for $f(t) = t^r P(t)$.

The conditions (A1),(A2),(B),(C),(D) give 10 linear equations in total.
Here, to exhibit the concrete form of the system of equations, we explain the case of the residue class
\[ r_3=r_4=r_5=0, \qquad R=0, \qquad r=60n .\]

In this case, (A1) gives the equation
\[c_0+c_1+c_2+c_3+c_4+c_5+c_6=1 . \]

Also, (A2) gives
\[60n+c_1+2c_2+3c_3+4c_4+5c_5+6c_6=0 . \]
Here we used (A1) to put $P(1)=1$ as in Section~\ref{sec:sympy-rigor}.

The condition (C) at a primitive 3rd root of unity gives
\[\begin{aligned} c_0+c_3+c_6-c_2-c_5&=1,\\ c_1+c_4-c_2-c_5&=0. \end{aligned}\]

The condition (B) at a primitive 4th root of unity gives
\[\begin{aligned} c_0-c_2+c_4-c_6&=1,\\ c_1-c_3+c_5&=0. \end{aligned}\]

The condition (D) at a primitive 5th root of unity gives
\[\begin{aligned} c_0-c_4+c_5&=1,\\ c_1-c_4+c_6&=0,\\ c_2-c_4&=0,\\ c_3-c_4&=0. \end{aligned}\]

Therefore, the linear system (10 equations) for this residue class is concretely
\[\left\{ \begin{aligned} c_0+c_1+c_2+c_3+c_4+c_5+c_6 &= 1,\\ 60n+c_1+2c_2+3c_3+4c_4+5c_5+6c_6 &= 0,\\ c_0+c_3+c_6-c_2-c_5 &= 1,\\ c_1+c_4-c_2-c_5 &= 0,\\ c_0-c_2+c_4-c_6 &= 1,\\ c_1-c_3+c_5 &= 0,\\ c_0-c_4+c_5 &= 1,\\ c_1-c_4+c_6 &= 0,\\ c_2-c_4 &= 0,\\ c_3-c_4 &= 0. \end{aligned} \right. \]

For the other 59 residue classes we also obtain linear systems of 10 equations with the same structure,
but the right-hand sides of the root-of-unity conditions change depending on $r_3,r_4,r_5$.
The script \cite{Span6} handles all of these 60 residue classes one by one.

There are 8 unknowns
\[(c_0,c_1,c_2,c_3,c_4,c_5,c_6,n) ,\]
and so the coefficient matrix is a $10\times8$ matrix.

Since only the right-hand sides depend on the residue class, the coefficient matrix is the same for all 60 residue classes, and we have verified that
\[\operatorname{rank}=8 .\]
On the other hand, since there are 10 equations, the linear relations among the rows form a 2-dimensional space.

Since the rank is equal to the number of unknowns, for each residue class the linear system falls into one of the following:

\begin{itemize}
  \item the case where it is inconsistent and has no solution,
  \item the case where it is consistent and has a unique numerical solution
  $(c_0,\dots,c_6,n)$.
\end{itemize}

In the actual computation, we see that 47 of the 60 residue classes are inconsistent,
and the other 13 give a unique numerical solution.

For these 13, we check whether $n$ is an integer,
and further check that each $c_k$ is an integer; then 9 remain.
For these, evaluating (A3)
\[f''(1)=0\]
directly, 7 remain.

These 7 solutions all have actual span 0,
and they also satisfy (A1),(A2), so
each of them is equal to
\[f(t)=1.\]

The computation above is carried out with the script \cite{Span6}.

\subsubsection{The case \texorpdfstring{$N=7$}{N=7}: relaxing the overdetermined system}

The linear equations obtained from (A1),(A2),(B),(C),(D) are 10 as in the previous case, but when we add the coefficient $c_7$, the unknown vector has the 9 unknowns
\[(c_0,\dots,c_7,n). \]

As in the case $N=6$, the $10\times9$ coefficient matrix is the same for all 60 residue classes, and we have verified that
\[\operatorname{rank}=9 .\]
Thus the linear relations among the rows now form a 1-dimensional space.
Since the rank is equal to the number of unknowns, each linear system is either inconsistent or has a unique numerical solution.

In the actual computation, more residue classes are consistent than in the case $N=6$:

\begin{itemize}
  \item Inconsistent: 32 of 60.
  \item Consistent with a unique solution: 28 of 60.
\end{itemize}

Of the 28 solutions, 14 satisfy $n\in\Z$ and all $c_k\in\Z$.
Imposing (A3) $f''(1)=0$ on these leaves 8, and all of them are $f(t)=1$.

The computation above is carried out with the script \cite{Span7}.

\subsubsection{The case \texorpdfstring{$N=8$}{N=8}: moving to a square system}

When we add the coefficient $c_8$, the unknown vector
\[(c_0,\dots,c_8,n)\]
has 10 unknowns, which is equal to the number of linear equations, 10.

As in the previous cases, the $10\times10$ coefficient matrix is the same for all 60 residue classes, and we have verified that
\[\det=1800 .\]
Therefore, the coefficient matrix is nonsingular,
and there is a unique numerical solution for each residue class.

Of these 60 solutions, 29 satisfy $n\in\Z$ and all $c_k\in\Z$.
Imposing (A3) on these leaves 9 integer solutions,
and all of them are the trivial solution $f(t)=1$.
The computation above is carried out with the script \cite{Span8}.

\begin{table}[htbp]
  \centering
  \caption{Linear structure and change in the solutions for the overdetermined and square systems ($N=6,7,8$)}
  \label{tab:overdetermined}
  \vspace{0.5em}
  \small
  \begin{tabular}{p{4.2cm}ccc}
    \toprule
    Item & $N=6$ & $N=7$ & $N=8$ \\
    \midrule
    Coefficient vector
      & $c_0,\dots,c_6$ (7)
      & $c_0,\dots,c_7$ (8)
      & $c_0,\dots,c_8$ (9) \\
    \texttt{linsolve} unknowns
      & $c_0,\dots,c_6,n$ (8)
      & $c_0,\dots,c_7,n$ (9)
      & $c_0,\dots,c_8,n$ (10) \\
    Rank of the coefficient matrix
      & 8
      & 9
      & 10 (square) \\
    Number of equations minus rank
      & 2
      & 1
      & 0 \\
    \midrule
    Residue classes with a consistent solution
      & 13 / 60
      & 28 / 60
      & 60 / 60 \\
    Of these, number with $n,c_k\in\Z$
      & 9
      & 14
      & 29 \\
    Number of solutions passing (A3)
      & 7
      & 8
      & 9 \\
    Genuine span-$N$ solutions
      & 0
      & 0
      & 0 \\
    \bottomrule
  \end{tabular}
\end{table}

Here, in the search for $N=6,7,8$, we impose explicitly only up to (A3).
However, all solutions that passed (A3) were $f(t)=1$.
Therefore, these solutions also satisfy (A4),(A5) automatically.

Hence, for $N=6,7,8$, the only solution that satisfies all the conditions (A1)--(A5),(B),(C),(D) is $f(t)=1$.

\subsection{Appearance of free variables and constraints by the higher derivative conditions (\texorpdfstring{$N=9,10,11$}{N=9,10,11})}

For $N\ge9$, the number of unknowns is larger than the number of linear equations, 10, and free variables appear in the linear stage.
For each $N=9,10,11$, we have verified that the linear system is consistent for all 60 residue classes, and that the free variables are
$n$ for $N=9$, $c_{10}$ and $n$ for $N=10$, and $c_{10},c_{11},n$ for $N=11$.
For these free variables, we apply the higher derivative conditions one after another.

\subsubsection{The case \texorpdfstring{$N=9$}{N=9}: the symbolic parameter \texorpdfstring{$n$}{n} and (A4)}

When $N=9$, there are 11 unknowns $(c_0,\dots,c_9,n)$, and for 10 equations exactly one free variable remains, namely $n$.
Thus, when we solve the 10 equations in the linear stage, each coefficient $c_k$ is uniquely determined as a
linear expression $c_k(n)$ in $n$ with rational coefficients.
We have verified that these linear expressions have integer coefficients for all 60 residue classes, that is,
$c_k(n)\in\Zn$ for $k=0,\dots,9$.
Hence $c_k(n_0)\in\Z$ for every integer $n_0$.

Substituting $c_k(n)$ into (A3) $f''(1)=0$ gives a polynomial equation
\[
g_2(n)=0
\]
in $n$ with integer coefficients, and we have verified that $g_2(n)$ is of degree exactly $2$ for all 60 residue classes.
Finding the integer roots of this equation exhaustively by the rational root theorem
(see item~(\ref{itm:4.2(2)}) in Section~\ref{subsec:AlgComp})
gives 12 integer root candidates in total over all residue classes.

Here, 2 of the 12 candidates satisfy $c_0\ne0$ and $c_9\ne0$, that is, they would give genuine span-9 solutions if they also satisfied (A4).
However, evaluating (A4) for them gives in fact
\[
f'''(1)=\pm360\ne0,
\]
so both are excluded.
Therefore, 2 of the 12 candidates are excluded by (A4),
and 10 solutions remain in the end.
All of these are the trivial solution $f(t)=1$.

The computation above is carried out with the script \cite{Span9}.

\subsubsection{The case \texorpdfstring{$N=10$}{N=10}: determining \texorpdfstring{$c_{10}$}{c10} by (A3)}

When $N=10$, the free variables in the linear stage are the two variables $c_{10}$ and $n$.

\begin{enumerate}

  \item
  \textbf{Determining $c_{10}$ by (A3)}:
  Substituting the expressions of $c_0,\dots,c_9$ in terms of $(c_{10},n)$ obtained in the linear stage into
  $f''(1)=0$ and collecting terms in $c_{10}$,
  we have verified that the coefficient of $c_{10}$ is $60$ for all 60 residue classes, independently of the value of $n$.
  Since this coefficient is a nonzero constant,
  $f''(1)=0$ is a linear equation in $c_{10}$, whose constant term is a polynomial in $n$, and
  $c_{10}$ is uniquely determined as a polynomial $c_{10}(n)$ in $n$ with rational coefficients.
  We have verified that the denominator $60$ cancels exactly, so that $c_{10}(n)\in\Zn$.
  Further substituting $c_{10}(n)$ into the expressions of $c_0,\dots,c_9$, we have also verified that
  $c_k(n)\in\Zn$ for all $k=0,\dots,10$.
  Note that the latter does not follow from the former, since the expressions of $c_0,\dots,c_9$ in terms of $(c_{10},n)$ may have rational coefficients.

  \item
  \textbf{Determining $n$ by (A4)}:
  With these $c_k(n)$, the condition (A4) $f'''(1)=0$
  becomes a polynomial equation
  \[
  g_3(n)=0
  \]
  in $n$ with integer coefficients, and we have verified that $g_3(n)$ is of degree exactly $3$ for all 60 residue classes.

\end{enumerate}

Finding the integer roots of $g_3(n)=0$ exhaustively by the rational root theorem
gives 11 integer roots in total over all residue classes.
Checking them directly, all of them give the trivial solution $f(t)=1$.

The computation above is carried out with the script \cite{Span10}.

\subsubsection{The case \texorpdfstring{$N=11$}{N=11}: determining two variables and (A5)}

When $N=11$, the free variables in the linear stage are the three variables $c_{10}$, $c_{11}$ and $n$.

\begin{enumerate}

  \item
  \textbf{Determining $c_{10},c_{11}$ by (A3),(A4)}:
  The conditions (A3) $f''(1)=0$ and (A4) $f'''(1)=0$ give two equations in $c_{10}$, $c_{11}$ and $n$, which are linear in $(c_{10},c_{11})$ with coefficients and constant terms in $\Q[n]$. We have verified that the determinant of the $2\times2$ coefficient matrix with respect to $(c_{10},c_{11})$ is $10800$ for all 60 residue classes, independently of $n$.
  Since this determinant is a nonzero constant, these two equations
  can be solved uniquely for $(c_{10},c_{11})$, and
  \[
  c_{10}=c_{10}(n),\qquad c_{11}=c_{11}(n)
  \]
  are uniquely determined as polynomials in $n$ with rational coefficients.
  We have verified that $c_{10}(n),c_{11}(n)\in\Zn$, and moreover that $c_k(n)\in\Zn$ for all $k=0,\dots,11$
  after substituting them into the expressions of $c_0,\dots,c_9$.

  \item
  \textbf{Determining $n$ by (A5)}:
  With these $c_k(n)$, the condition (A5) $f''''(1)=0$
  becomes a polynomial equation
  \[
  g_4(n)=0
  \]
  in $n$ with integer coefficients, and we have verified that $g_4(n)$ is of degree exactly $4$ for all 60 residue classes.

\end{enumerate}

Finding the integer roots of $g_4(n)=0$ exhaustively by the rational root theorem
gives 12 integer roots in total over all residue classes.
Checking them directly, all of them give the trivial solution $f(t)=1$.

The computation above is carried out with the script \cite{Span11}.

\begin{table}[htbp]
  \centering
  \caption{Algebraic structure and role of the derivative conditions in the cases
  with free variables ($N=9,10,11$)}
  \label{tab:underdetermined}
  \vspace{0.5em}
  \small
  \begin{tabularx}{\textwidth}{>{\raggedright\arraybackslash}p{3.0cm}*{3}{>{\raggedright\arraybackslash}X}}
    \toprule
    Item & $N=9$ & $N=10$ & $N=11$ \\
    \midrule
    Coefficient vector
      & $c_0,\dots,c_9$ (10)
      & $c_0,\dots,c_{10}$ (11)
      & $c_0,\dots,c_{11}$ (12) \\
    Free variables in the linear stage
      & $\{n\}$ (1)
      & $\{c_{10},n\}$ (2)
      & $\{c_{10},c_{11},n\}$ (3) \\
    \midrule
    Role of (A3)
      & constraint on $n$ (degree 2)
      & determining $c_{10}$ (coefficient 60)
      & determining $c_{10},c_{11}$ with (A4) (determinant 10800) \\
    Role of (A4)
      & excluding candidates
      & constraint on $n$ (degree 3)
      & determining $c_{10},c_{11}$ with (A3) \\
    Role of (A5)
      & not needed
      & not needed
      & constraint on $n$ (degree 4) \\
    \midrule
    Final number of integer solutions
      & 10
      & 11
      & 12 \\
    Genuine span-$N$ solutions
      & 0
      & 0
      & 0 \\
    \bottomrule
  \end{tabularx}
\end{table}

\subsection{Independent re-check}

For all integer solutions obtained, to confirm the correctness of the algebraic processing,
we carried out an independent re-check.
In this re-check, we did not reuse the general derivative function based on the Leibniz formula
used in the search stage; instead we constructed the actual Laurent polynomial $f(t)=t^rP(t)$,
differentiated it directly with the SymPy function \texttt{diff},
substituted $t=1,i,\omega,\zeta$, and evaluated directly the conditions imposed in each case:
\begin{itemize}
  \item for $N=6,7,8$, the 6 conditions (A1)--(A3),(B),(C),(D),
  \item for $N=9,10$, the 7 conditions (A1)--(A4),(B),(C),(D),
  \item for $N=11$, the 8 conditions (A1)--(A5),(B),(C),(D).
\end{itemize}
For all solutions found, these independent checks were successful.

Note that this re-check confirms that the solutions found indeed satisfy the conditions.
The fact that no solution is missed follows from the exhaustiveness of the linear stage
over all 60 residue classes and of the search for integer roots by the rational root theorem.

\subsection{Conclusion}

In the search explained so far, the derivative conditions used are increased step by step depending on $N$.
For $N=6,7,8$, we use up to (A3), and since all solutions obtained at that stage are
$f(t)=1$, (A4),(A5) are also satisfied automatically.
For $N=9,10$, we use up to (A4), and since all solutions obtained in the end are
$f(t)=1$, (A5) is also satisfied automatically.
For $N=11$, we apply up to (A5) explicitly.
As a result, in the range $N\le11$, that is,
in the range $\operatorname{span}(f)\le11$,
there is no nontrivial Laurent polynomial with integer coefficients
that satisfies all the conditions (A1)--(A5),(B),(C),(D).

The search above gives a rigorous computer-assisted proof in the sense explained in Section~\ref{subsec:AlgComp},
and this proves Theorem~\ref{thm:poly}.

Finally, we complete the proof of Theorem~\ref{thm:main}.
Suppose that a knot $K$ admits purely cosmetic surgeries.
Then, by \eqref{eq:j1}, \eqref{eq:j2}, \eqref{eq:j3}, \eqref{eq:cs1}, \eqref{eq:cs2} and Lemma~\ref{lem:arf_i},
its Jones polynomial $V_K(t)$ satisfies the conditions (A1)--(A5),(B),(C),(D).
Hence, if $\operatorname{span}(V_K(t))\le 11$, then $V_K(t)=1$ by Theorem~\ref{thm:poly}.
This completes the proof of Theorem~\ref{thm:main}.

\section{Nontrivial solutions for span 12}\label{sec:span12}

In the previous sections, we confirmed that there is no nontrivial solution with span at most 11.
In this section, we see that nontrivial solutions actually appear for $N=12$.

\subsection{Search method}

Let $N=12$ and consider
\[
P(t)=c_0+c_1t+\cdots+c_{12}t^{12}, \qquad f(t)=t^rP(t),
\]
with $r=60n+R$ as before.
By the same method as in Section~\ref{sec:algorithm}, using the conditions (A1)--(A5),(B),(C),(D),
the coefficients $c_k$ are uniquely determined as polynomials in $n$ with integer coefficients for each residue class $R$.
In particular, every integer $n$ gives an integer solution.

We computed these solutions for $n\in[-30,30]$ and all 60 residue classes, that is, for $60\times61=3660$ cases.
Among them, 3647 solutions have actual span exactly $12$, that is, satisfy $c_0\ne0$ and $c_{12}\ne0$.
We compared $\max_{0\le k\le12}\vert c_k\vert$ for these 3647 solutions.
The correctness of the solution given below does not depend on this search, since we verify it directly in the following subsections.

\subsection{Solution with the smallest coefficients}

One of the solutions obtained corresponds to $r=1$, that is, $R=1$ and $n=0$
(so $(r_3,r_4,r_5)=(1,1,1)$), and its coefficients are
\[
(c_0,c_1,\ldots,c_{12}) = (3,-4,5,-6,6,-7,7,-6,6,-5,4,-3,1).
\]
Therefore
\[\begin{aligned} P(t) ={}&t^{12}-3t^{11}+4t^{10}-5t^9+6t^8-6t^7+7t^6-7t^5\\ &\quad+6t^4-6t^3+5t^2-4t+3, \end{aligned}\]
and
\[\begin{aligned} f(t) ={}&t^{13}-3t^{12}+4t^{11}-5t^{10}+6t^9-6t^8+7t^7-7t^6\\ &\quad+6t^5-6t^4+5t^3-4t^2+3t. \end{aligned}\]

For this solution, we have
\[
\sum_{k=0}^{12}c_k=1, \qquad \max_k\vert c_k\vert=7, \qquad \sum_{k=0}^{12}\vert c_k\vert=63.
\]
Since $c_0=3\ne0$ and $c_{12}=1\ne0$, the actual span of $f(t)$ is $12$.

Among the 3647 solutions of actual span $12$ found in this finite search (with $n\in[-30,30]$), this solution has the smallest value of $\max_k\vert c_k\vert$.
The only other solution with the same value is its mirror image $f(t^{-1})$, which corresponds to $r=-13$, that is, $R=47$ and $n=-1$.

\subsection{Independent check of the derivative conditions}

For this solution, we check (A1)--(A5) by differentiating the Laurent polynomial $f(t)$ directly.
We have
\[\begin{aligned} f'(t) ={}&13t^{12}-36t^{11}+44t^{10}-50t^9+54t^8-48t^7\\ &\quad+49t^6-42t^5+30t^4-24t^3+15t^2-8t+3, \end{aligned}\]
\[\begin{aligned} f''(t) ={}&2\bigl( 78t^{11}-198t^{10}+220t^9-225t^8+216t^7\\ &\qquad\quad-168t^6+147t^5-105t^4+60t^3-36t^2+15t-4 \bigr), \end{aligned}\]
\[\begin{aligned} f'''(t) ={}&6\bigl( 286t^{10}-660t^9+660t^8-600t^7+504t^6\\ &\qquad\quad-336t^5+245t^4-140t^3+60t^2-24t+5 \bigr), \end{aligned}\]
\[\begin{aligned} f''''(t) ={}&24\bigl( 715t^9-1485t^8+1320t^7-1050t^6+756t^5\\ &\qquad\quad-420t^4+245t^3-105t^2+30t-6 \bigr). \end{aligned}\]
Substituting $t=1$ gives
\[
f(1)=1, \qquad f'(1)=f''(1)=f'''(1)=f''''(1)=0.
\]
Therefore, (A1)--(A5) all hold.

\subsection{Check of the root-of-unity conditions}

Recall that the minimal polynomials of $\omega$, $i$, $\zeta$ are the cyclotomic polynomials
\[
\Phi_3(t)=t^2+t+1,\qquad
\Phi_4(t)=t^2+1,\qquad
\Phi_5(t)=t^4+t^3+t^2+t+1 .
\]
Dividing $f(t)$ by these polynomials, we have
\[
\begin{aligned}
f(t)
={}&(t^2+t+1)
\bigl(t^{11}-4t^{10}+7t^9-8t^8+7t^7-5t^6\\
&\qquad\qquad\quad
+5t^5-7t^4+8t^3-7t^2+4t-1\bigr)+1,
\end{aligned}
\]
\[
\begin{aligned}
f(t)
={}&(t^2+1)
\bigl(t^{11}-3t^{10}+3t^9-2t^8+3t^7-4t^6\\
&\qquad\qquad\quad
+4t^5-3t^4+2t^3-3t^2+3t-1\bigr)+1,
\end{aligned}
\]
\[
\begin{aligned}
f(t)
={}&(t^4+t^3+t^2+t+1)
\bigl(t^9-4t^8+7t^7-9t^6+11t^5\\
&\qquad\qquad\quad
-11t^4+9t^3-7t^2+4t-1\bigr)+1.
\end{aligned}
\]
That is, $f(t)\equiv 1 \pmod{\Phi_m(t)}$ for $m=3,4,5$.
Since $\Phi_3(\omega)=\Phi_4(i)=\Phi_5(\zeta)=0$, we obtain
\[
f(\omega)=f(i)=f(\zeta)=1 .
\]
Therefore, (B), (C), (D) all hold.

\begin{remark}
The solution above and its mirror image are excluded by another known condition.
By Daemi--Lidman--Miller Eismeier \cite[Corollary 1.4]{DLME}, the Alexander polynomial of a knot $K$ admitting purely cosmetic surgeries is trivial,
and so $\vert V_K(-1)\vert=\det(K)=1$ (see also \cite{KanenobuKishimotoSumi}).
On the other hand, the solution above satisfies $f(-1)=-63$.
Hence it cannot be the Jones polynomial of a knot admitting purely cosmetic surgeries.
This shows that the conditions (A1)--(A5),(B),(C),(D) alone are not enough for span 12,
and that further conditions, such as $\vert f(-1)\vert=1$, are needed to extend our method.
We leave this to future work.
\end{remark}

\section*{Statement on the use of AI tools}

In this study, AI tools (Claude, ChatGPT, and Gemini) were used to generate the initial versions of the SymPy scripts used in the search (\cite{Span6,Span7,Span8,Span9,Span10,Span11}) and to check them. 
They were also used to make the initial draft of this manuscript.
However, the author has independently checked all of the code, the computational results, and the text, and the author takes full final responsibility for the content of this paper.

\bibliographystyle{amsplain}
\bibliography{PCSjones}

\end{document}